\documentclass[12pt]{article}
\usepackage{caption,cite}

\usepackage{amsthm,amsmath,amssymb}
\usepackage{verbatim}
\usepackage[colorlinks=true,citecolor=black,linkcolor=black,urlcolor=blue]{hyperref}
\newtheorem{theorem}{Theorem}
\newtheorem{lemma}[theorem]{Lemma}

\theoremstyle{definition}

\newtheorem{conjecture}[theorem]{Conjecture}

\theoremstyle{remark}

\newcommand{\h}{\mathcal{H}}

\title{\bf An Erd\H{o}s-Gallai type theorem for uniform hypergraphs}
\author{
Akbar Davoodi\thanks{
School of Mathematics, Institute for Research in Fundamental Sciences (IPM),
P.O.Box: 19395-5746, Tehran, Iran, 
E-mail address: davoodi@ipm.ir
} \qquad Ervin Gy\H{o}ri \thanks{MTA Renyi Institute/ Dept. of Mathematics, Central European University (Budapest), E-mail address: gyori.ervin@renyi.mta.hu} \qquad Abhishek Methuku \thanks{Dept. of Mathematics, Central European University (Budapest), E-mail address: abhishekmethuku@gmail.com} \qquad Casey Tompkins \thanks{MTA Renyi Institute (Budapest), E-mail address:  ctompkins496@gmail.com}\\[4pt]
}

\date{}

\begin{document}

\maketitle



\begin{abstract}
A well-known theorem of Erd\H{o}s and Gallai \cite{Er-Ga} asserts that a graph with no path of length $k$ contains at most $\frac{1}{2}(k-1)n$ edges.  Recently Gy\H{o}ri, Katona and Lemons \cite{GyoKaLe} gave an extension of this result to hypergraphs by determining the maximum number of hyperedges in an $r$-uniform hypergraph containing no Berge path of length $k$ for all values of $r$ and $k$ except for $k=r+1$.  We settle the remaining case by proving that an $r$-uniform hypergraph with more than $n$ edges must contain a Berge path of length $r+1$.
\end{abstract}

Given a hypergraph $\h$, we denote the vertex and edge sets of $\h$ by $V(\h)$ and $E(\h)$ respectively.   Moreover, let $e(\h) = \left| E(\h) \right|$ and $n(\h) =\left| V(\h) \right| $.

A \emph{Berge path} of length $k$ is a collection of $k$ distinct hyperedges $e_1,\dots,e_k$ and $k+1$ distinct vertices $v_1,\dots, v_{k+1}$ such that for each $1 \le i \le k$, we have $v_i,v_{i+1} \in e_i$.  A \emph{Berge cycle} of length $k$ is a collection of $k$ distinct hyperedges $e_1,\dots,e_k$ and $k$ distinct vertices $v_1,\dots, v_{k}$ such that for each  $1 \le i \le k-1$, we have $v_i,v_{i+1} \in e_i$ and $v_k,v_1 \in e_k$.  The vertices $v_i$ and edges $e_i$ in the preceding definitions are called the vertices and edges of their respective Berge path (cycle).  The Berge path is said to start at the vertex $v_1$.  We also say that the edges $e_1,\dots,e_k$ of the Berge path (cycle) span the set $\cup_{i=1}^k e_i$.  

A hypergraph is called $r$-uniform, if all of its hyperedges have size $r$.
Gy\H{o}ri, Katona and Lemons determined the largest number of hyperedges possible in an $r$-uniform hypergraph without a Berge path of length $k$ for both the range $k>r+1$ and the range $k \le r$.
\begin{theorem}[Gy\H{o}ri--Katona--Lemons, \cite{GyoKaLe}]
Let $\h$ be an $r$-uniform hypergraph with no Berge path of length $k$.  If $k>r+1>3$, we have
\begin{displaymath}
e(\h) \le \frac{n}{k} \binom{k}{r}.
\end{displaymath}
If $r \ge k>2$, we have
\begin{displaymath}
e(\h) \le \frac{n(k-1)}{r+1}.
\end{displaymath}
\end{theorem}
The case when $k=r+1$ remained unsolved. Gy\H{o}ri, Katona and Lemons conjectured that the upper bound in this case should have the same form as the $k>r+1$ case:

\begin{conjecture}[Gy\H{o}ri--Katona--Lemons, \cite{GyoKaLe}] Fix $k = r+1>2$ and let $\h$ be an $r$-uniform hypergraph containing no Berge path of length $k$.  Then,
\begin{displaymath}
e(\h) \le \frac{n}{k} \binom{k}{r} = n.
\end{displaymath}
\end{conjecture}
In this note we settle their conjecture by proving
\begin{theorem}
\label{main}
Let $\h$ be an $r$-uniform hypergraph.   If $e(\h) > n$, then $\h$ contains a Berge path of length at least $r+1$.
\end{theorem}
A construction with a matching lower bound when $r+1$ divides $n$ is given by disjoint complete hypergraphs on $r+1$ vertices.
Observe that by induction it suffices to prove Theorem~\ref{main} when the hypergraph is connected.  We will prove the following stronger theorem.

\begin{theorem}
\label{strongerversion}
Let $\h$ be a connected $r$-uniform hypergraph. If $e(\h)\geq n$, then for every vertex $v\in V(\h)$ either there exists a Berge path of length $r+1$ starting from $v$ or there exists a Berge cycle of length $r+1$ with $v$ as one of its vertices.
\end{theorem}

To see that Theorem \ref{strongerversion} implies Theorem \ref{main}, suppose $e(\h)>n$ and assume that after applying Theorem \ref{strongerversion} we find a Berge cycle of length $r+1$. If the Berge cycle is not the complete $r$-uniform hypergraph on $r+1$ vertices, then its edges span a vertex which is not a vertex of the Berge cycle. Starting from this vertex and then using all of the edges of the Berge cycle would yield a Berge path of length $r+1$. If the Berge cycle is a complete hypergraph, then by connectivity and the assumption $e(\h)>n$, there must be another hyperedge which intersects it, and we may find a Berge path of length $r+1$ again.

We will need the following Lemma in the proof of Theorem \ref{strongerversion}.
\begin{lemma}
\label{cyclelemma}
Let $v$ be a vertex and $e$ be an edge in a hypergraph $\h$ with $v \in e$.  Consider a Berge cycle of length $r$ with vertices $\{v_1,\dots,v_r\}$ and edges $\{e_1,\dots,e_r\}$ such that $v \not\in \{v_1,\dots,v_r\}$ and $e \not\in \{e_1,\dots,e_r\}$ and assume that it spans a set $X$ of vertices such that $X \cap (e\setminus \{v\}) \neq \varnothing$.  Then, there is a Berge path of length $r+1$ starting at $v$ or a Berge cycle of length $r+1$ containing $v$.
\end{lemma}
\begin{proof}
First, suppose that $X \cap (e\setminus \{v\})$ contains a vertex $u \not\in \{v_1,\dots,v_r\}$.  Without loss of generality, let $u \in e_1$.  Then, we have the Berge path $v,e,u,e_1,v_2,e_2,\dots,v_r,e_r,v_1$ of length $r+1$ starting at $v$.  Now suppose $X \cap (e\setminus \{v\}) \subset \{v_1,\dots,v_r\}$, and assume without loss of generality that $v_1 \in X \cap (e\setminus \{v\})$.  Consider the edges $e_1$ and $e_r$. If either contains an element not in $\{v_1,\dots,v_r,v\}$, then we will find a Berge path of length $r+1$.  Indeed, suppose $u$ is such an element and $u \in e_r$, then we have the Berge path $v,e,v_1,e_1,v_2,e_2,\dots,v_r,e_r,u$.  Finally, if neither $e_1$ nor $e_r$ contains elements outside of $\{v_1,\dots,v_r,v\}$, then since they are distinct sets at least one of them contains $v$, say $e_r$.  We can then find a Berge cycle of length $r+1$ with $v$ as a vertex, namely $v,e,v_1,e_1,v_2,e_2,\dots,v_r,e_r,v$.
\end{proof}

\begin{proof}[Proof of Theorem \ref{strongerversion}]

We will use induction first on $r$, and for each $r$, on $n$.  First, we prove the statement for graph case ($r=2$). Let $G$ be a graph and fix a vertex $v\in V(G)$.  Consider a  breadth-first search spanning tree $T$ with root $v$. If there is no path of length three starting from $v$, then $T$ has two levels, $N_1(v)$ and $N_2(v)$. By assumption $G$ has at least $n$ edges. Hence, $G$ has at least one more edge than $T$. If both ends of this edge belong to $N_1(v)$, then we have a triangle containing $v$. Otherwise, it is easy to see that there is a path of length three starting from $v$.

Now, let $\h= (V,E)$ be a connected $r$-uniform hypergraph with $r \ge 3$ and let $v \in V(\h)$ be an arbitrary vertex. 

First, suppose that there is a cut vertex $v_0$, that is, the (non-uniform) hypergraph $\h' = (V',E')$ where $V' = V \setminus \{v_0\}$ and $E' = \{e \setminus \{v_0\}: e \in E\}$ is not connected.  In this case, let the connected components be $C_1,\dots,C_s$, and for each $i$, let $\h_i$ be the hypergraph attained by adding back $v_0$ to the edges in $C_i$.   At least one of these $\h_i$'s, say $\h_1$ satisfies the conditions of the theorem since, if $e(\h_i) \le n(\h_i)-1$ for all $i$, then $$e(\h) = \sum_{i=1}^s e(\h_i) \le  \sum_{i=1}^s n(\h_i) - s = n(\h) -1,$$
a contradiction. If $v \in V(\h_1)$ (this includes the case when $v=v_0$), then we are done by applying induction to $\h_1$.  Assume $v \neq v_0$ and let $v \in V(\h_i)$, $i \neq 1$, then by induction, $\h_1$ contains a Berge path of length $r$ starting from $v_0$ (as a Berge cycle of length $r+1$ with $v_0$ as a vertex yields a Berge path of length $r$ starting at $v_0$), and since $\h_i$ contains a Berge path from $v$ to $v_0$, their union is a Berge path of length at least $ r+1 $ starting at $v$, as desired. Therefore, from now on we may assume there is no cut vertex in $\h$, so in particular $v$ is not a cut vertex.

Let $e\in E(\h)$ be an edge containing $v$ and let $\h'$ be the hypergraph defined by removing $e$ from the edge set of $\h$ and deleting $v$ from all remaining edges in $\h$.
Let $C_1, \ldots, C_s$, $s\geq 1$ be the connected components of $\h'$ and observe that each of them contains a vertex of $e\setminus \{v\}$.
By the pigeonhole principle there is some component $C_i$ such that $e(C_i)\geq n(C_i)$.
In order to apply the induction hypothesis, we will replace the $r$-edges in the component $C_i$ by edges of size $r-1$ in such a way that no multiple edges are created and the component remains connected.  We proceed by considering one $r$-edge at a time and attempting to remove an arbitrary vertex from it.

Suppose for some $r$-edge, say $f$, this is not possible.  If for every vertex $u$ in $f$, replacing $f$ with $f\setminus\{u\}$ disconnects the hypergraph, then every hyperedge which intersects $f$ intersects it in only one point, and hyperedges which intersect $f$ in different points will be in different components if we delete $f$.  Let $F_1,F_2,\dots,F_r$ be the connected components in $C_i$ obtained from deleting $f$.  Then, by the pigeonhole principle we find a component $F_j$ with $e(F_j) \ge n(F_j)$ and continue the procedure on that component instead.

Thus, we may assume there exists some vertex of $f$ whose removal from $f$ does not disconnect the hypergraph. Now, consider the case when the deletion of any vertex of $f$ would lead to multiple edges in the hypergraph. This means that every $r-1$ subset of $f$ is already an edge of the hypergraph. Clearly, in this case there is a Berge cycle of length $r$ using each vertex of $f$. In the original hypergraph, if this Berge cycle spans a vertex of $e \setminus \{v\}$, then it can be extended  to a Berge path of length $r+1$ starting from $v$ or a Berge cycle of length $r+1$ with $v$ as one of its vertices by Lemma \ref{cyclelemma}. If it does not span a vertex of $e \setminus \{v\}$, then there is a Berge path of length at least two from $v$ to the Berge cycle which, in turn, can easily be extended to a Berge path of length $r+1$.

We may now assume that $f$ contains at least one element whose removal does not disconnect the hypergraph and at least one element whose removal does not create a multiple edge.  If there is an element $w$ such that removing $w$ from $f$ disconnects the hypergraph, then no element of $f \setminus \{w\}$ will yield a multiple edge if deleted (for then $w$ would not disconnect the hypergraph) and so we can find an element to remove from $f$.  If there is no such element $w$ whose removal disconnects the hypergraph, we are also done since we can simply take any element of $f$ whose removal does not make a multiple edge.

Therefore, we can transform $C_i$ into an $(r-1)$-uniform and connected hypergraph $\h^*$ satisfying $e(\h^*)\geq n (\h^*)$. By the induction hypothesis, for every vertex $z\in V(\h^*)$ there exists a Berge path of length $r$ starting from $z$ or there exists a Berge cycle of length $r$ containing $z$. Choose $z$ to be in the edge $e$.  The associated Berge path (or cycle) in original hypergraph is a Berge path (or cycle) of the same length. If the result is a Berge path, then we are done trivially by extending it with $e$ and $v$.  If the result is a Berge cycle, then we are done by Lemma \ref{cyclelemma}.
\end{proof}

\section*{Acknowledgment}
The first author's research was supported by a grant from IPM.
The research of the second, third and fourth authors is partially supported by the National Research, Development and Innovation Office NKFIH, grant K116769.


\begin{thebibliography}{10}

\bibitem{Er-Ga} Paul Erd\H os,  Tibor Gallai,  On maximal paths and
circuits of graphs. {\it Acta Math. Acad. Sci. Hungar.}  \textbf{10}, (1959)
337-356.

\bibitem{GyoKaLe} Ervin Gy\H ori, Gyula Y. Katona, Nathan Lemons,
Hypergraph extensions of the Erd\H os-Gallai
Theorem, {\it European Journal of Combinatorics} \textbf{58}, (2016) 238-246.

\end{thebibliography}
\end{document}